\documentclass{article}
\usepackage{amsmath}
\usepackage{amsthm}
\usepackage{amssymb}

\theoremstyle{plain}
\newtheorem{thm}{Theorem}
\newtheorem{prop}[thm]{Propositon}
\newtheorem{cor}[thm]{Corollary}

\theoremstyle{definition}

\newtheorem{exmp}[thm]{Example}

\theoremstyle{remark}

\def\AA{\mathcal{A}}

\def\NN{\mathcal{N}}
\def\RR{\mathcal{R}}

\title{
Natural generalized invertibility and prescribed idempotents
}
\author{G. Kant\'un-Montiel\footnote{This research was partially supported by CONACYT, Mexico, under grant CB-2011-01.}}
\date{}

\begin{document}
\maketitle
\begin{abstract}
We study the natural inverse introduced by X. Mary and show some connections with the $(p,q)$-inverses of Djordjevic and Wei, where $p$ and $q$ are prescribed idempotents. We deal first with rings with identity and then specialize to the particular case of the algebra of bounded linear operators. We give a characterization of the set of operators along which an operator is natural invertible
in terms of prescribed range and nullspace. Finally, the special case when the prescribed idempotent is the spectral projection is discussed.
\end{abstract}



\section{Introduction}

Several generalizations of invertibility, such as Moore-Penrose, Drazin and group inverses, are special types of outer inverses. More recently, Xavier Mary has introduced a class of outer inverses (\cite{Ma11}), which is defined below. Subsequent sections are devoted to its study.
While there are useful applications of these generalized inverses for matrices and bounded linear operators, working in the more general setting of rings can sometimes lead to a better understanding. 
In this paper we will work first in the setting of rings and then specialize to the operators case. 

Let $\AA$ be a ring with identity $1$. Recall an element $p\in \AA$ is called idempotent if $p=p^2$.

Let $a\in \AA$. As usual, $a$ is invertible  if there exists $b\in \AA$ such that $ab=ba=1$. We denote the  inverse of $a$ by $a^{-1}$.

We call $a$ an inner regular element if there exists $b\in \AA$ such that
\begin{equation} \label{eqnIR}
a=aba.
\end{equation}
Also, the element $b$ is called an inner inverse for $a$ and we will denote it by $a^{(1)}$. Note that $ab$ and $ba$ are idempotents.

If for $a$ there exists $b\in \AA$ satisfying
\begin{equation} \label{eqnOR}
b=bab,
\end{equation}
then we say that $a$ is outer regular and $b$ is an outer inverse for $a$, denoted by $a^{(2)}$. Of course, if $a$ is outer regular with outer inverse $b$, then $b$ is inner regular. 
If (\ref{eqnIR}) and (\ref{eqnOR}) holds, then we say that $b$ is a reflexive inverse for $a$, which will be denoted by $a^{(1,2)}$. It is easy to check that if $a$ is inner regular, then $bab$ is a reflexive inverse for $a$. 

Neither inner, outer or reflexive inverses are unique. For uniqueness we have to require commutativity: an element $a\in \AA$ is group invertible if there exists $b\in \AA$ such that
\begin{equation}
a=aba, \quad b=bab, \quad ab=ba.
\end{equation}
We denote the group inverse of $a$ by $a^\sharp$. The group inverse is unique if it exists. 

As we pointed out above, if $a$ is inner regular, then it is outer regular, so further generalizations should weaken inner regularity while requiring outer regularity. An element $a\in \AA$ is Drazin invertible if for some $n\in \mathbb{N}$ there exists $b\in \AA$ such that
\begin{equation} \label{eqnDrazin}
a^n=a^nba, \quad b=bab, \quad ab=ba.
\end{equation} 
We have that $a$ is Drazin invertible if and only if $a^n$ is group invertible for some $n\in \mathbb{N}$. The Drazin inverse is unique if it exists.

 Interested as we are in having uniqueness, we still want to drop inner invertibility. An element $a\in \AA$ is $(p,q)$-invertible, in the sense of Djordjevic and Wei, if there exist $b\in \AA$ and idempotents $p, q\in \AA$ such that
\begin{equation}
bab= b, \quad ba=p, \quad 1-ab=q.
\end{equation}
The $(p,q)$-inverse is unique if it exists and we will refer to it as the Djordjevic-Wei $(p,q)$-inverse, denoted by $a^{(2)}_{p,q}$.

The interested reader may find more information on the inverses defined above in \cite{DjRBook}.

For left principal ideals, we write $a\AA \subset b\AA$ if there exists $x\in \AA$ such that $a=bx$. In the same way, $a\AA  \supset b\AA $ if there exists $y\in \AA$ such that $ay=b$. Finally, $a\AA = b \AA$ if $a\AA \subset b\AA$ and $a\AA \supset b\AA$. The case of right principal ideals is analogous.

An element $a\in \AA$ is natural invertible along $d\in \AA$, in the sense of Mary, if there exists $b\in \AA$ such that
\begin{equation}
bab=b, \quad b\AA = d \AA, \quad \AA b = \AA d
\end{equation}
From now on we will write $a^{\parallel d}$ for the natural inverse of $a$ along $d$ and refer to it as the Mary inverse of $a$ along $d$. It turns out that the Mary inverse is unique if it exists.

In this paper we characterize, in the next section,  Mary invertible elements in terms of Djordjevic-Wei $(p,q)$-invertibility. Then, in Section 3, we deal with the particular case of the algebra of bounded  linear operators. Finally, in section 4, we discuss a little on spectral theory.

\section{Outer inverses with prescribed idempotents}

In a series of papers, X. Mary investigated the properties of the natural inverse along an element and proved the following:
\begin{thm}[{\cite[Theorem 7]{Ma11}}] \label{thmMa}
An element $a\in \AA$ is Mary invertible along $d$ if and only if $\AA d\subset \AA ad$ and $(ad)^\sharp$ exists.
In this case, $a^{\parallel d}=d(ad)^\sharp$.
\end{thm}

On the other hand, Djordjevic and Wei proved:
\begin{thm}[{\cite[Theorem 2.2]{DjW05}}] \label{thmDjW}
Let $a,c\in \AA$, and suppose $p,q\in \AA$ are idempotents such that there exist $a^{(2)}_{p,q}$ and $c^{(1,2)}_{1-q,1-p}$. Then $ac$ is group invertible and $a^{(2)}_{p,q}=c(ac)^\sharp$.
\end{thm}

Notice that the existance of $c^{(1,2)}_{1-q,1-p}$ means that there exists $t\in \AA$ such that $c=ctc$, $t=tct$, $ct=p$ and $tc=1-q$.


We can characterize
the Mary inverse as an outer inverse with prescribed idempotents.

\begin{thm} \label{thmMaryDjW}
Let $\AA$ be a ring and $a,b,d\in \AA$. The following statements are equivalent:
\begin{enumerate} 
\item $b$ is the Mary inverse of $a$ along $d$. \label{thmPIa}
\item $b$ is an outer inverse for $a$, $d$ is an inner regular element and there exists a reflexive inverse $t\in \AA$ of $d$ such that $ba=dt$ and $ab=td$. \label{thmPIb}
\end{enumerate}
\end{thm}
\begin{proof}
$(\ref{thmPIa}) \Rightarrow (\ref{thmPIb})$
Suppose $b=a^{\parallel d}$. From Theorem \ref{thmMa} we have that $b=d(ad)^\sharp$ and it follows $b=bab$. Let $t=(ad)^\sharp a$, then
$$tdt=(ad)^\sharp ad (ad)^\sharp a = (ad)^\sharp a=t,$$
$$dtd=d(ad)^\sharp ad=bad=d,$$
$$ba=d(ad)^\sharp a=dt,$$
$$ab= ad (ad)^\sharp =(ad)^\sharp ad =td.$$
Hence, \ref{thmPIb}. is satisfied.

$(\ref{thmPIb}) \Rightarrow (\ref{thmPIa})$
Suppose $b$ is an outer inverse for $a$ and $t$ is a reflexive inverse for $d$ such that $ba=dt$ and $td=ab$. From Theorem \ref{thmDjW} we have that $b=d(ad)^\sharp$. Now, by Theorem \ref{thmMa}, we only have to show $\AA d \subset \AA ad$. But this is true since
$$d=dtd=dab=dad(ad)^\sharp=d(ad)^\sharp ad.$$
Hence, $b$ is the Mary inverse of $a$ along $d$.
\end{proof}

\begin{cor}
If $a$ is Mary invertible along $d$, then it is Djordjevic-Wei $(p,q)$-invertible with $p=dt$ and $q=1-td$ where $t=(ad)^\sharp a$.
\end{cor}
\begin{proof}
Suppose $a$ is Mary invertible along $d$, and let $b=a^{\parallel d}$. By Theorem \ref{thmMaryDjW}, $b=a^{2}_{dt,1-td}$ and $a$ is Djordjevic-Wei $(dt,1-td)$-invertible.

\end{proof}

\begin{exmp}
Let $\text{Mat}_2$ be the ring of $2\times 2$ matrices with real entries, and consider $A,D_a\in \text{Mat}_2$
$$A:=
\begin{bmatrix}
0&1\\
0&0
\end{bmatrix},
\qquad
D_a:=
\begin{bmatrix}
2a&0\\
a&0
\end{bmatrix},
$$
with $a\in \mathbb{R}$, $a\neq 0$ fixed.
It is clear that $A$ is not invertible, and using Theorem \ref{thmMa} we see that $A$ is Mary invertible along $D_a$.
Indeed,  a direct computation shows that 
$(AD_a)^\sharp = 
\left[ \begin{smallmatrix}
1/a&0\\
0&0
\end{smallmatrix} \right]
 $
and if 
$S:=\left[ \begin{smallmatrix}
2&0\\
1&0
\end{smallmatrix} \right]
 $
, we have $D_a=SAD_a$.

Notice that 
the Djordjevic-Wei projections $P,Q\in \text{Mat}_2$ are the same for every choice of $a\neq 0$.
Indeed, by the proof of the theorem above, let $T:=(AD_a)^\sharp A=
\left[ \begin{smallmatrix}
1&0\\
0&0
\end{smallmatrix} \right]
$
, then $T$ is a reflexive inverse for $D_a$, and $P=D_aT=
\left[ \begin{smallmatrix}
0&2\\
0&1
\end{smallmatrix} \right]
$ and $Q=I-TD_a=
\left[ \begin{smallmatrix}
0&0\\
0&1
\end{smallmatrix} \right]
$.

$D_a$ is not an outer inverse for $A$. In fact, the Mary inverse along $D_a$ for any choice of $a\neq 0$ is
$$
B=D(AD)^\sharp=
\begin{bmatrix}
2&0\\
1&0
\end{bmatrix},
$$
which coincides with the Djordjevic-Wei $(P,Q)$-inverse (see \cite[Theorem 2.1]{DjW05}).
\end{exmp}


\section{Outer inverses with prescribed range and nullspace}

Let $X$ be a Banach space.
If $M$ and $N$ are  subspaces of $X$ such that $M\cap N=\{0\}$ and $M+N=X$, we  say that $M$ is complemented with complement $N$ (also $N$ is complemented with complement $M$). If in addition $M$ and $N$ are closed, then we write  $X=M\oplus N$.

Let $B(X)$ denote the set (algebra) of bounded linear operators on $X$. 
For an operator $A\in B(X)$, we will denote by $\RR(A)$ the range of $A$ and by $\NN(A)$ the nullspace of $A$.

Any nonzero operator has a nonzero outer generalized inverse.
On the other hand,
an operator $A\in B(X)$ is inner regular if and only if $\NN(A)$ and $\RR(A)$ are closed and complemented subspaces of $X$.
In this case, if $M$ and $N$ are such that
$X=N\oplus \NN(A)$ and $X=\RR(A)\oplus M$, then $A$ has the following matrix form:
$$A= 
\begin{bmatrix}
A_1&0\\
0&0
\end{bmatrix}
: 
\begin{bmatrix}
N\\
\NN(A)
\end{bmatrix}
\rightarrow
\begin{bmatrix}
\RR(A)\\
M
\end{bmatrix},
$$
where $A_1$ is invertible. Furthermore, if $B\in B(X)$ is an inner inverse of $A$ such that $\RR(BA)=N$ and $\NN(AB)=M$, then $B$ has the following matrix form:
$$B= 
\begin{bmatrix}
A_1^{-1}&0\\
0&W
\end{bmatrix}
: 
\begin{bmatrix}
\RR(A)\\
M
\end{bmatrix}
\rightarrow
\begin{bmatrix}
N\\
\NN(A)
\end{bmatrix}.
$$
where $W$ is an arbitrary bounded linear operator from $M$ to $\NN(A)$. Thus, an inner inverse is not unique even if we fix its range and nullspace. However, a reflexive inverse is uniquely determined by an appropiate choice of range and nullspace.



Now we show that the Mary inverse of an operator is an outer inverse with prescribed range and nullspace.

\begin{thm} \label{thmMaryOuter}
Let $A,T\in B(X)$ be nonzero operators. The following statements are equivalent:
\begin{enumerate}
\item $B$ is the Mary inverse of $A$ along $T$. \label{thmMO1}
\item $B$ is an outer inverse of $A$ such that $\RR(B)=\RR(T)$ and $\NN(B)=\NN(T)$. \label{thmMO2}
\end{enumerate}
\end{thm}
\begin{proof}
$\ref{thmMO1}\Rightarrow \ref{thmMO2}.$
Suppose $A$ is Mary invertible along $T$ with Mary inverse $B$. Then $B$ is an outer inverse for $A$ and there exist  $L,U,V,W\in B(X)$ such that 
$$B=LT, \quad T=VB,$$
$$B=TU, \quad T=BW.$$
Then, 
$$\RR(T)=\RR(BW)\subseteq \RR(B)=\RR(TU)\subseteq \RR(T),$$
hence $\RR(T)=\RR(B)$.
On the other hand,
$$\NN(T)\subseteq \NN(LT) =\NN(B)\subseteq \NN(VB)=\NN(T),$$
and $\NN(T)=\NN(B)$.

$\ref{thmMO2}\Rightarrow \ref{thmMO1}.$
Suppose that $B$ is an outer inverse for $A$ and $T$ is such that $\RR(B)=\RR(T)$ and $\NN(B)=\NN(T)$.
Since $B$ is inner regular, there exist closed subspaces $M,N\subset X$ such that
$X=N\oplus \NN(B)$ and $X=\RR(B)\oplus M$, and we have the following matrix form for $B$:
$$B= 
\begin{bmatrix}
B_1&0\\
0&0
\end{bmatrix}
: 
\begin{bmatrix}
N\\
\NN(B)
\end{bmatrix}
\rightarrow
\begin{bmatrix}
\RR(B)\\
M
\end{bmatrix}
$$
with $B_1$ invertible.
Also, since $\RR(T)=\RR(B)$ and $\NN(B)=\NN(T)$, we have that $\NN(T)$ and $\RR(T)$ are closed and complemented subspaces. Thus, $T$ is inner regular and have the following matrix form
with respect to the same decomposition of spaces as above:
$$T= 
\begin{bmatrix}
T_1&0\\
0&0
\end{bmatrix}
: 
\begin{bmatrix}
N\\
\NN(B)
\end{bmatrix}
\rightarrow
\begin{bmatrix}
\RR(B)\\
M
\end{bmatrix}
$$
with $T_1$ invertible.
Now, let the operators $L,V\in B(\RR(B)\oplus M)$ and $U,W\in B(N\oplus\NN(B))$ be defined by
$L= \left[
\begin{smallmatrix}
B_1T_1^{-1}&0\\
0&L_2
\end{smallmatrix}
\right]$
,
$U= \left[
\begin{smallmatrix}
T_1^{-1}B_1&0\\
0&U_2
\end{smallmatrix}
\right]$
,
$V= \left[
\begin{smallmatrix}
T_1B_1^{-1}&0\\
0&V_2
\end{smallmatrix}
\right]$
and
$W=\left[
\begin{smallmatrix}
B_1^{-1}T_1&0\\
0&W_2
\end{smallmatrix}
\right]$, where $L_2,V_2\in B(M)$ and $U_2,W_2\in B(\NN(B))$ are arbitrary operators.
Then, a simple calculation shows that
$LT=B$, $TU=B$, $VB=T$ and $BW=T$.
\end{proof}

Note that the operators $L,U,V,W$ in the proof of the theorem above may not be unique. 
However, from the unicity of the Mary inverse, the operator $B$ is unique. 

\begin{exmp}
Let $X=\ell^2(\mathbb{N})$ the space of square-summable sequences, and let $A,T\in B(X)$ be defined by
$$Ax:=(\tfrac{1}{2}x_2, \tfrac{1}{3}x_3, \tfrac{1}{4}x_4, \ldots),$$
$$Tx:=(0,x_2,x_1,0,0,\ldots).$$
for $x=(x_1,x_2,x_3,x_4,\ldots)$.
Then it is easy to check that $AT$ is group invertible and $(AT)^\sharp x=(3x_2, 2x_1, 0,0, \ldots)$. Also, if $S\in B(X)$ is defined by
$$Sx:=(0,2x_1, 3x_2,0,0,\ldots),$$
we see that $T=SAT$, thus $B(X)T\subset B(X)AT$, and from Theorem \ref{thmMa}, $A$ is Mary invertible along $T$ with Mary inverse $B=T(AT)^\sharp$. We have
$$Bx=(0,2x_1, 3x_2, 0,0,\ldots)$$

Now, define an operator $V\in B(X)$ by
$$Vx:=(0,x_1,x_2,0,0, \ldots),$$
so $\RR(B)=\RR(V)$ and $\NN(B)=\NN(V)$. By the theorem above, $A$ is Mary invertible along $V$ with Mary inverse $B$. Indeed, a straight computation shows that $AV$ is group invertible and $(AV)^\sharp x= (2x_1, 3x_2,0,0,\ldots)$. Also, if $W\in B(X)$ is defined by
$$Wx:=(0,2x_1, 3x_2,0,0,\ldots)$$
we get that $V=WAV$, and from Theorem \ref{thmMa}, $A$ is Mary invertible along $V$  and indeed we have
$$V(AV)^\sharp x=(0,2x_1,3x_2,0,0,\ldots ) = Bx.$$
\end{exmp}

Recall a bounded projection, or simply a projection, is an idempotent operator with closed range. That is, $P\in B(X)$ is a projection if $P=P^2$ and $\RR(P)$ is closed.

Now we give a characterization of the set of operators along which an operator $A$ is Mary invertible.

\begin{thm} \label{thmMaryChar}
Let $A,T\in B(X)$ be nonzero operators. The following statements are equivalent.
\begin{enumerate}
\item $A$ is Mary invertible along $T$.
\item $\RR(T)$ and $\NN(T)$ are closed and complemented subspaces of $X$, $A(\RR(T))=\RR(AT)$ is closed, $\RR(AT)\oplus \NN(T)=X$ and the reduction $A|_{\RR(T)}:\RR(T)\rightarrow \RR(AT)$ is invertible.
\end{enumerate}
\end{thm}
\begin{proof}
Suppose $A$ is Mary invertible along $T$ with Mary inverse $B\in B(X)$. Then, from Theorem \ref{thmMaryOuter}, $B$ is an outer inverse for $A$ such that $\RR(B)=\RR(T)$ and $\NN(B)=\NN(T)$. Since $A$ is an inner inverse for $B$, $\RR(B)$ and $\NN(B)$ (and thus $\RR(T)$ and $\NN(T)$) are closed and complemented subspaces of $X$. Furthermore, $I-AB$ is a projection from $X$ on $\NN(B)=\NN(T)$, thus $X=\RR(AB)\oplus \NN(T)$, and since $\RR(AB)=A(\RR(B))=A(\RR(T))=\RR(AT)$ we have that $\RR(AT)$ is closed and $X=\RR(AT)\oplus \NN(T)$.
Now, for the invertibility of $A|_{\RR(T)}:\RR(T)\rightarrow \RR(AT)$  it is clear that it is onto. To see that $A|_{\RR(T)}$ is also $1-1$ on $\RR(T)$, suppose that there exists $x\in \RR(T)$ such that $Ax=0$. Since $x\in \RR(T)=\RR(B)$, there exists $y\in X$ such that $By=x$. Then we have $0=Ax=BAx=BABy=By$ and thus $x=0$. Therefore $A_{\RR(T)}$ is $1-1$ and onto, and hence invertible.

Conversely, suppose that $\RR(T)$ and $\NN(T)$ are closed and complemented subspaces of $X$, $X=\RR(AT)\oplus \NN(T)$, and the reduction $A|_{\RR(T)}:\RR(T)\rightarrow \RR(AT)$ is invertible. Let $M$ be the complement of $\RR(T)$, so $X=\RR(T)\oplus M$. Then $A$ has the following matrix form with respect to these decompositions of spaces:
\begin{equation}\label{eqnMatrixA}
A=
\begin{bmatrix}
A_1&A_3\\
A_4&A_2
\end{bmatrix}
:
\begin{bmatrix}
\RR(T)\\
M
\end{bmatrix}
\rightarrow
\begin{bmatrix}
\RR(AT)\\
\NN(T)
\end{bmatrix}
.
\end{equation}
Since $A$ maps $\RR(T)$ to $\RR(AT)$, and $A_4:\RR(T)\rightarrow \NN(T)$, it follows that $A_4=0$ and we have that $A_1=A|_{\RR(T)}$ is invertible.
Now, let $B$ be the operator defined by
\begin{equation}\label{eqnMatrixB}
B=
\begin{bmatrix}
A_1^{-1}&0\\
0&0
\end{bmatrix}
:
\begin{bmatrix}
\RR(AT)\\
\NN(T)
\end{bmatrix}
\rightarrow
\begin{bmatrix}
\RR(T)\\
M
\end{bmatrix}
.
\end{equation}
A direct verification shows that $BAB=B$, $\RR(B)=\RR(T)$ and $\NN(B)=\NN(T)$. Thus, by Theorem \ref{thmMaryOuter}, $B$ is the Mary inverse of $A$ along $T$. Therefore, $A$ is Mary invertible along $T$.
\end{proof}

If $A$ is Mary invertible along $T$, from the proof of the previous theorem we know that $A$ has the matrix form of (\ref{eqnMatrixA}) with $A_4=0$. We claim that also $A_3=0$. Indeed, again from the proof of the theorem above, we have that if $B$ is the Mary inverse of $A$ along $T$, then $B$ has the matrix form of (\ref{eqnMatrixB}). Since $BA$ is the projection from $X$ on $\RR(B)=\RR(T)$, from the matrix form
$$
BA=
\begin{bmatrix}
I&A_1^{-1}A_3\\
0&0
\end{bmatrix}
:
\begin{bmatrix}
\RR(T)\\
M
\end{bmatrix}
\rightarrow
\begin{bmatrix}
\RR(T)\\
M
\end{bmatrix}
,
$$
we see that $A_1^{-1}A_3=0$, and it follows that $A_3=0$. Hence, we have the following:

\begin{cor}
Let $A$ be Mary invertible along $T$. Then $A$ has the following matrix form:
\begin{equation}\label{eqnMatrixAF}
A=
\begin{bmatrix}
A_1&0\\
0&A_2
\end{bmatrix}
:
\begin{bmatrix}
\RR(T)\\
M
\end{bmatrix}
\rightarrow
\begin{bmatrix}
\RR(AT)\\
\NN(T)
\end{bmatrix}
,
\end{equation}
with $A_1$ invertible.
\end{cor}

The Mary inverse along an element of a semigroup has been studied in a series of papers by Xavier Mary and Pedro Patricio. It is clear that the Mary inverse generalizes the usual inverse, and it was proved in \cite{Ma12} that this inverse generalizes the group, Moore-Penrose and Drazin inverse.

Let $S$ be a *-semigroup and $a\in S$. Recall the Moore-Penrose inverse of $a$ is an element $b$ such that
$$aba=a, \quad bab=b, \quad (ab)^*=ab, \quad (ba)^*=ba.$$
The Moore-Penrose inverse is unique if it exists, and will be denoted by $a^\dag$.

\begin{thm}[\cite{Ma12}, Theorem 11] Let $S$ be a semigroup and $a\in S$.
\begin{enumerate}
\item $a$ is group invertible if and only if it is Mary invertible along $a$. In this case the Mary inverse along $a$ is inner and coincides with the group inverse.
\item $a$ is Drazin invertible if and only if it is Mary invertible along some $a^m$, $m\in \mathbb{N}$, and in this case the two inverses coincide.
\item If $S$ is a *-semigroup, $a$ is Moore-Penrose invertible if and only if it is invertible along $a^*$. In this case the inverse along $a^*$ is inner and coincides with the Moore-Penrose inverse.
\end{enumerate}
\end{thm}

Let $H$ be a Hilbert space and $B(H)$ the algebra of bounded linear operators on $H$.
Recall that we denote by $A^\sharp$, $A^d$ and $A^\dag$ the group, Drazin, and Moore-Penrose inverse of $A$, respectively. Using Theorem \ref{thmMaryOuter} we get at once the following:

\begin{cor}
Let $A$ be Mary invertible along $T$ with Mary inverse $B$.
\begin{enumerate}
\item If $T$ is invertible, $B=A^{-1}$.
\item If $\RR(T)=\RR(A)$, $\NN(T)=\NN(A)$, then $B=A^\sharp$.
\item If $\RR(T)=\RR(A^n)$, $\NN(T)=(A^n)$, then $B=A^d$.
\item If $A,T\in B(H)$, $\RR(T)=\RR(A^*)$, $\NN(T)=\NN(A^*)$, then $B=A^\dag$.
\end{enumerate}
\end{cor}


For instance, for a group invertible operator $A$,  since $\RR(A)=\RR(A^\sharp)=\RR(A^\sharp A)=\RR(A A^\sharp)$ and $\NN(A)=\NN(A^\sharp)=\NN(A^\sharp A)=\NN(A A^\sharp)$, we have that the Mary inverse of $A$ along $A$, $A^\sharp$, $A^\sharp A$ or $AA^\sharp$ is the group inverse.

\section{
Spectral sets}


Recall the spectrum of an operator $A\in B(X)$ is the set
$$\sigma(A):=\{\lambda \in \mathbb{C}: A-\lambda I \text{ is not invertible}\},$$
the resolvent set is $\rho(A):=\mathbb{C}\setminus \sigma(A)$ and for $\lambda\in \rho(A)$ the resolvent function is
$$R_\lambda(A):=(A-\lambda I)^{-1}.$$


A subset $\Lambda \subset \sigma(A)$ is an spectral set if $\Lambda$ and $\sigma(A)\setminus \Lambda$ are both closed in $\mathbb{C}$. For a spectral set $\Lambda$ for $A$, the spectral projection associated with $A$ and $\Lambda$ is defined by
$$P_\Lambda(A):=-\frac{1}{2\pi i}\int_C R_\lambda(A)d\lambda,$$
where $C$ is a Cauchy contour that separates $\Lambda$ from $\sigma(A)\setminus\Lambda$. 


Let $Q\in B(X)$. If $\sigma(Q)=\{0\}$, then we say that $Q$ is a quasinilpotent operator. Recall that $Q$ is nilpotent if $Q^n=O$ for some $n\in \mathbb{N}$, and nilpotent operators are quasinilpotent.

Let $A\in B(X)$, the quasinilpotent part of $A$ is the set
$$H_0(A):=\{x\in X: \lim_{n\rightarrow \infty}\|A^nx\|^{1/n}\}=0.$$
Of course, $A$ is quasinilpotent if and only if $H_0(A)=X$.

The analytical core of $A$ is the set $K(A)$ of all $x\in X$ such that there exists a sequence $(u_n)\subset X$ and a constant $c>0$ such that:
\begin{enumerate}
\item $x=u_0$, and $Au_{n+1}=u_n$ for every $n\in \mathbb{Z}_+$,
\item $\|u_n\|\leq c^n\|x\|$ for every $n\in \mathbb{Z}_+$.
\end{enumerate}

We have that $K(A)$ is a  subspace of $X$ and $A(K(A))=K(A)$.

If $0$ is a point of the resolvent set or an isolated point of the spectrum $\sigma(A)$, then the operator $A$ is called quasipolar. 
Let $A$ be quasipolar and let $P_0$ be the spectral projection associated with the spectral set $\{0\}$, then \cite[Theorem 3.74]{AienaBook}:
$$\RR(P_0)=H_0(A)
\quad \text{and}\quad
\NN(P_0)=K(A).$$
Thus, $H_0(A)$ and $K(A)$ are closed and we have the following decomposition:
$$X=K(A)\oplus H_0(A).$$


Quasipolar operators are generalized invertible in the sense of Koliha:
an operator $A\in B(X)$ is Koliha-Drazin invertible if there exists $B\in B(X)$ such that
\begin{equation}
A(I-AB) \text{ is quasinilpotent}, \quad B=BAB, \quad AB=BA.
\end{equation}

An operator $A$ is Koliha-Drazin invertible if and only if $0$ is isolated in $\sigma(A)$.
If $0\in \sigma(A)$ is a pole of the resolvent of order $n$, then $A$ is Drazin invertible with Drazin index $n$. If $0$ is a simple pole then it is group invertible.

Now we can offer another proof that the Mary inverse generalizes the Koliha-Drazin inverse (cf. \cite[Theorem 8 and Theorem 12]{Ma12}).

\begin{prop}
$A$ is Koliha-Drazin invertible if and only if it is Mary invertible along $T$, where $T$ is such that $\RR(T)=K(A)$ and $\NN(T)=H_0(A)$. Furthermore, the two inverses coincide if they exist.
\end{prop}
\begin{proof}
Suppose $A$ is Koliha-Drazin invertible. Then $0$ is an isolated point of the spectrum, thus $K(A)$ and $H_0(A)$ are closed and $X=K(A)\oplus H_0(A)$. Let $T$ be the projection onto $K(A)$ parallell to $H_0(A)$. Then, $\RR(T)=K(A)$ and $\NN(T)=H_0(A)$ are closed and complemented, $\RR(AT)=A(K(A))=K(A)$ is closed, $\RR(AT)\oplus \NN(T)=K(A)\oplus H_0(A)=X$, and since $A(K(A))=K(A)$ (\cite[Theorem 1.21(ii)]{AienaBook}) and $\NN(A)\cap K(A)=\{0\}$ (\cite[Remark 2.4 (d) and Theorem 2.22(iii)]{AienaBook}), we have that $A|_{\RR(T)}$ is invertible. Therefore, from Theorem \ref{thmMaryChar} it follows that $A$ is Mary invertible along $T$.

Conversely, suppose $A$ is Mary invertible along $T$, with $\RR(T)=K(A)$ and $\NN(T)=H_0(A)$. From Theorem \ref{thmMaryChar} it follows that $K(A)$ and $H_0(A)$ are closed. Since $\RR(AT)=A(\RR(T))=A(K(A))=K(A)$, also from Theorem \ref{thmMaryChar}, we have that $X=K(A)\oplus H_0(A)$. Then, from \cite[Theorem 3.76]{AienaBook}, $0$ is an isolated point of $\sigma(A)$. Therefore $A$ is Koliha-Drazin invertible.

Now  suppose that $A$ is Koliha-Drazin invertible and also Mary invertible along $T$. Since $A$ is Koliha-Drazin invertible, $0$ is an isolated point of $\sigma(A)$. Let $P$ be the spectral projection associated with the spectral set $\{0\}$. From \cite[Theorem 3.74]{AienaBook} we have that $\RR(P)=H_0(A)$ and $\NN(P)=K(A)$. Now, $A$ and $B=A^{\parallel T}$ have the following  matrix form (\ref{eqnMatrixB},\ref{eqnMatrixAF}):
$$
A=
\begin{bmatrix}
A_1&0\\
0&A_2
\end{bmatrix}
:
\begin{bmatrix}
K(A)=\RR(I-P)\\
H_0(A)=\NN(I-P)
\end{bmatrix}
\rightarrow
\begin{bmatrix}
K(A)\\
H_0(A)
\end{bmatrix}
,
$$

$$
B=
\begin{bmatrix}
A_1^{-1}&0\\
0&0
\end{bmatrix}
:
\begin{bmatrix}
K(A)=\RR(I-P)\\
H_0(A)=\NN(I-P)
\end{bmatrix}
\rightarrow
\begin{bmatrix}
K(A)\\
H_0(A)
\end{bmatrix}
$$
Which is the same as the matrix form of the Koliha-Drazin inverse.
Therefore, $A^{\parallel T }=A^D$.


\end{proof}



As noted above, the Koliha-Drazin is a particular case when we consider the spectral set $\Lambda=\{0\}$. For the general case when $\Lambda$ is a spectral set such that $0\in \Lambda$, Dajic and Koliha have defined a generalized inverse and studied its properties \cite{DajicKoliha}.

\begin{thm}
Let $A\in B(X)$ and $\Lambda$ an spectral set for $A$. If $0\notin \Lambda$ then $A$ is Mary invertible along $P_\Lambda(A)$.
\end{thm}
\begin{proof}
Let $P=P_\Lambda(A)$. Then $\RR(P)$ and $\NN(P)$ are closed and $X=\RR(P)\oplus \NN(P)$.
Now, since $\RR(P)$ is $A$-invariant, $\sigma(A|_{\RR(P)})=\Lambda$ and $0\notin \Lambda$ we have that
$A|_{\RR(P)}:\RR(P)\rightarrow \RR(P)$
is invertible.
Thus, $\RR(AP)=\RR(P)$ is closed, $\RR(AP)\oplus N(P) = X$ and $A|_{\RR(P)}:\RR(P)\rightarrow \RR(AP)$ is invertible.
Therefore, by Theorem \ref{thmMaryChar}, $A$ is Mary invertible along $P$.
\end{proof}

\begin{cor}
Let $A\in B(X)$ and $\Lambda$ an spectral set for $A$. If $0\in \Lambda$ then $A$ is Mary invertible along $I-P_\Lambda(A)$.
\end{cor}
\begin{proof}
If $0\in \Lambda$, then $0\notin \sigma(A)\setminus\Lambda$. From the theorem above, $A$ is Mary invertible along $P_{\sigma(A)\setminus\Lambda}(A)=I-P_\Lambda (A)$.
\end{proof}

It was shown in \cite{Ma12} that if a projection $P$ commutes with $A$ and $A$ is Mary  invertible along $P$, then $\RR(P)\subset K(T)$.
In fact, for $r>0$, $r\in \mathbb{R}$, let 
$$K_r(A):=\{x_0\in X:(\exists x_n)(\forall n\in \mathbb{N}) Ax_n=x_{n-1}, \lim\sup_{n\rightarrow\infty}\|x_n\|^{1/n}<r^{-1} \},$$
 Koliha and Poon have shown that if $\Lambda=\{\lambda:|\lambda |<r\}\cap \sigma(A)$ is a spectral set 
such that
the circle  $C_r=\{\lambda:|\lambda|=r\}$
separates $\Lambda$ from $\sigma(A)\setminus \Lambda$
, then $\NN(P_\Lambda (A))=K_r(A)\subset K(A)$ \cite[Theorem 5.6]{KolihaPoon}.
Of course, from the corollary above, $A$ is Mary invertible along $I-P_\Lambda(A)$.

%


\noindent Gabriel Kant\'un-Montiel \\
Centro de Investigaci\'on en Matem\'aticas, \\ 
Jalisco S/N, Valenciana 36240, \\
Guanajuato, Gto., M\'exico. \\
Email: gabriel.kantun@cimat.mx

\end{document}